\documentclass[10 pt,conference]{ieeeconf}  
\IEEEoverridecommandlockouts                              
\usepackage{blkarray,multirow}
\usepackage{booktabs}
\usepackage{algorithm}
\usepackage{algpseudocode}
\algnewcommand{\Inputs}[1]{%
  \State \textbf{Inputs:}
  \Statex \hspace*{\algorithmicindent}\parbox[t]{.8\linewidth}{\raggedright #1}
}
\algnewcommand{\Initialize}[1]{%
  \State \textbf{Initialize:}
  \Statex \hspace*{\algorithmicindent}\parbox[t]{.8\linewidth}{\raggedright #1}
}
\algnewcommand{\Outputs}[1]{%
  \State \textbf{Outputs:}
  \Statex \hspace*{\algorithmicindent}\parbox[t]{.8\linewidth}{\raggedright #1}
}
\usepackage{graphicx,epsfig}
\usepackage{epstopdf}
\usepackage{tikz}
\usetikzlibrary{arrows,positioning} 
\tikzset{
    >=stealth',
    block/.style={
           rectangle,
           rounded corners,
           draw=black, very thick,
           text width=10em,
           minimum height=5em,
           text centered},
    link/.style={
           ->,
           thick,
           shorten <=2pt,
           shorten >=2pt,}
}

\usepackage{amssymb,amsmath,amsfonts,amsthm}

\theoremstyle{plain}
\newtheorem{thm}{Theorem}
\newtheorem{lem}{Lemma}

\theoremstyle{definition}
\newtheorem{defn}{Definition}

\newtheorem{assump}{Assumption}

\theoremstyle{remark}

\newcommand{\iden}{\mathbf I}

\newcommand{\A}{\mathcal A}

\newcommand{\node}{\mathcal N}
\newcommand{\E}{\mathcal E}
\newcommand{\G}{\mathcal G}
\newcommand{\N}{\mathbb N}
\renewcommand{\S}{\mathcal S}
\newcommand{\R}{\mathbb R}
\renewcommand{\P}{\mathcal P}
\DeclareMathOperator{\st}{s.t.}

\DeclareMathOperator{\tr}{\bf trace}

\DeclareMathOperator{\Real}{\bf Re}
\DeclareMathOperator{\diag}{\bf diag}
\DeclareMathOperator*{\argmax}{arg\,max}
\DeclareMathOperator*{\argmin}{arg\,min}

\DeclareMathOperator*{\interior}{int}

\makeatletter
\newcommand{\pushright}[1]{\ifmeasuring@#1\else\omit\hfill$\displaystyle#1$\fi\ignorespaces}
\newcommand{\pushleft}[1]{\ifmeasuring@#1\else\omit$\displaystyle#1$\hfill\fi\ignorespaces}
\renewcommand*\env@matrix[1][*\c@MaxMatrixCols c]{%
  \hskip -\arraycolsep
  \let\@ifnextchar\new@ifnextchar
  \array{#1}}
\makeatother

\title{\LARGE \bf On Resilience Analysis and Quantification for Wide-Area Control of Power Systems}
\author{Yueyun Lu, Chin-Yao Chang, Wei Zhang, Laurentiu D. Marinovici and Antonio J. Conejo
\thanks{Y. Lu, C.-Y. Chang are with the Department of Electrical and Computer Engineering, The Ohio State University, Columbus, OH 43210, USA}
\thanks{W. Zhang is with the Department of Electrical and Computer Engineering, The Ohio State University, Columbus, OH 43210, USA, with a joint appointment in the Electricity Infrastructure Group, Pacific Northwest National Laboratory, Richland, WA 99354, USA}
\thanks{L. D. Marinovici is with the Electricity Infrastructure Group, Pacific Northwest National Laboratory, Richland, WA 99354, USA}
\thanks{A. J. Conejo is with the Departments of Electrical and Computer Engineering and the Department of Integrated Systems Engineering, The Ohio State University, Columbus, OH 43210, USA}
\thanks{This work was funded by Laboratory Directed Research and Development funding under the Control of Complex Systems Initiative at Pacific Northwest National Laboratory, which is operated for the US Department of Energy by Battelle Memorial Institute under Contract DE-AC05-76RL01830.}}

\begin{document}

\maketitle
\thispagestyle{empty}
\pagestyle{empty}

\begin{abstract}
Wide-area control is an effective mean to reduce inter-area oscillations of large power systems. Its dependence on communication of remote measurement signals makes the closed-loop system vulnerable to cyber attacks. This paper develops a framework to analyze and quantify resilience of a given wide-area controller under disruptive attacks on certain communication links. Resilience of a given controller is measured in terms of closed-loop eigenvalues under the worst possible attack strategy. The computation of such a resilience index is challenging especially for large-scale power systems due to the discrete nature of the attack strategies. To address the challenge, we propose an optimization-based formulation and a convex relaxation approach to facilitate the computation. Conditions under which the relaxation is exact are derived and an efficient algorithm with guaranteed convergence is also developed. The proposed framework and the algorithm allow us to quantify resilience for given wide-area controllers and also provide sufficient conditions to guarantee closed-loop stability under all possible communication attacks. Simulations are performed on the IEEE 39-bus system to illustrate the proposed resilience analysis and computation framework.
\end{abstract}

\section{Introduction}
With the power grid increasingly working close to its operation limit, inter-area oscillation becomes ever more lightly damped, which easily results in instability~\cite{venkatasubramanian2004analysis}. Local decentralized controllers, such as power system stabilizers (PSSs), are designed to suppress local oscillations. They may interact in an adverse way, if not carefully tuned, that aggravates inter-area oscillations. Motivated by the advancement in the Wide-Area Measurement System (WAMS) technology, recent research efforts have been focusing on wide-area control (WAC) problems~\cite{chakrabortty2013introduction,kamwa2001wide,chaudhuri2004wide}. The goal of WAC is to achieve better closed-loop performance, such as inter-area oscillation damping, by the use of remote measurement signals via the Phasor Measurement Units (PMUs) installed across the grid.

One important class of literature on WAC is concerned with optimal control design under certain performance metric. The main control objective is inter-area oscillation damping, for which various metrics have been proposed. In the design of supplementary damping controller (SDC) using Linear Parameter Varying (LPV) model~\cite{liu2006lpv}, the metric is given by the signal amplification from disturbance to output. To design FACTS (Flexible AC
Transmission Systems)-based control facilitated by an aggregate model~\cite{chakrabortty2012wide}, the metric is defined on the closed-loop transient response of inter-area oscillation modes. A mixed $H_2/H_{\infty}$ output feedback control design is studied in~\cite{zhang2008design} where the metric is concerned with geometric measures of modal controllability/observability. Another control objective is voltage stability. For the automatic scheduling and coordination of voltage control devices~\cite{vu1996improved,zobian1996steady,tomsovic2005designing}, the metric is composed of several terms regarding switching cost, penalty on voltage violations and penalty on circular VAR flow. Typically, the controllers are designed for a fixed structure, that is to say, the communication network has a pre-specified structure. There has been a recent interest in incorporating communication structure into the design. Due to the fact that most optimal control formulations result in controllers without any sparsity pattern and require centralized implementation, a sparsity-promoting optimal control scheme is proposed in~\cite{dorfler2014sparsity} where the $\ell_1$ regularization term in the objective accounts for the structural design. 


Another body of literature is concerned with delays and failures arising in the communication network of WAMS. To deal with network delays, a predictor-based $\mathcal{H}_{\infty}$ control design strategy is discussed in~\cite{chaudhuri2004wide} to account for a delayed arrival of feedback signals. Furthermore, an arbitration approach is proposed in~\cite{soudbakhsh2015delay} to exploit the flexibility of communication network so that the designed controllers are in sync with network delays, making the closed-loop system delay-aware, rather than just delay-tolerant. To counteract the impact of communication failures on the closed-loop system, a framework proposed in~\cite{zhang2014wide} utilizes a hierarchical set of wide-area measurements for feedback and employs channel switching based on mathematical morphology identification.

Existing works on WAC resilience mostly focus on communication delays or failures. There has been limited discussion on resilience under adversaries. Due to the increasing threat on cyber security~\cite{mo2012cyber,teixeira2015secure}, remote signal transmission via communication channels is prone to cyber attacks. As WAC relies heavily on the availability of remote signals, the integrity of communication network plays a crucial role in the closed-loop performance. In this paper, we consider the adversary has disruptive resources~\cite{teixeira2015secure} that can result in unavailability of the signals transmitted over communication channels. Such an attack model is commonly referred to as Denial of Service (DoS) attack~\cite{amin2009safe}. To launch a DoS attack, the adversary can jam the communication channels, compromise devices and prevent them from sending data, attack the routing protocols, flood network traffic, among others. Our goal is to develop a framework to analyze and quantify resilience under DoS attacks. In particular, we aim to design effective ways to determine whether a given wide-area controller is resilient, and how resilient it is under certain attack strategy. To achieve this, we use network-reduced linearized power system model under linear feedback control. Such a model is widely used in the literature on WAC problems~\cite{liu2006lpv,tomsovic2005designing,chakrabortty2012wide,zima2005design,zhang2008design}. We first define resilience in terms of closed-loop spectral abscissa (the largest real part of eigenvalues) under the worst possible attack strategy. The direct computation of such a resilience metric is challenging, especially in large-scale network due to its combinatorial nature. We then propose an equivalent optimization-based formulation and a convex relaxation approach to facilitate the computation. On the theoretic side, we derive a condition under which the relaxation is exact. On the practical side, we develop an efficient algorithm for the relaxed problem with guaranteed convergence. The algorithm not only provides resilience criterion but also reveals structural vulnerabilities. These results contribute new perspectives to WAC with an emphasis on resilience under DoS communication attacks. They also allow us to systematically analyze resilience properties of a given wide-area controller.

\section{Problem Formulation}
\label{sec:prob}

In this paper, we consider a network-reduced power system model commonly used in the literature~\cite{nabavi2015distributed,soudbakhsh2015delay,liu2006lpv,tomsovic2005designing,chakrabortty2012wide,zima2005design,zhang2008design}. The overall power system is represented by an interconnected dynamical system defined on a graph $\G=(\node,\E)$, where $\node\triangleq\{1,\cdots,N\}$ denotes the set of buses and $\E$ denotes the set of transmissions lines between buses. Let $x_i(t)\in\R^{n_i}$ be the state variables associated with bus $i$. Depending on the level of details used in the generator model, $x_i$ can represent generator phase angle, frequency, quadrature-axis internal emf, state variables of Power System Stabilizer (PSS) or other local controllers. Typically, local dynamics and local controllers can be described by linear systems subject to nonlinear coupling terms due to power exchange with neighboring buses. The overall system can be written is the following form:
\begin{align*}
\dot{x}_i=A_{ii}x_i+c_i+\sum_{(i,j)\in\E,j\neq i}h(x_i,x_j),
\end{align*}
where $A_{ii}\in\R^{n_i\times n_i}$ is the system matrix that has incorporated local controls, $c_i$ is a constant term regarding mechanical power input and $h(x_i,x_j)$ is a nonlinear function representing the power flow between buses $i$ and $j$. By linearization at a stationary operating point, we arrive at a distributed control system of the following form:
\begin{align}
	\dot{x}_i=A_{ii}x_i+\sum_{j\in\node,j\neq i} A_{ij}x_j+B_iu_i,\quad i\in\node, \label{eq:subsysi}
\end{align}
where with slight abuse of notation, $x_i$ represents the deviation of state variables from the nominal operating point, $A_{ij}$ captures the linearized coupling between buses $i$ and $j$ ($A_{ij}=0$ if there is no coupling), and $B_iu_i$ is an introduced wide-area control action that reacts to deviations from the nominal operating point based on both local and remote state information. We consider wide-area control $u_i$ to be composed of local component $u_{i,loc}$ that depends on local state information and wide-area component $u_{i,wac}$ that depends on remote state information in the following form:
\begin{align}
	u_i=u_{i,loc}+u_{i,wac}=K_{ii}x_i+\sum_{j\in\node,j\neq i} K_{ij}x_j, \label{eq:control}
\end{align}
where $K_{ij}\in\R^{m_i\times n_j},i,j\in\node$ are feedback gains. The local component $u_{i,loc}$ is an additional correction on top of local controllers, which can be set to zero if there is no such correction. Note that the sparsity pattern of feedback gains captures the structure of communication network. Define $n\triangleq \sum_{i=1}^N n_i,m\triangleq \sum_{i=1}^N m_i$. Let $x=[x_1^T,\cdots,x_N^T]^T\in\R^n$ and $u=[u_1^T,\cdots,u_N^T]^T\in\R^m$. The overall system can be described by
\begin{align}
	\dot{x}(t)=(A+BK)x(t), \label{eq:cl}
\end{align}
where $A=[A_{ij}]_{1\le i,j\le N}\in\R^{n\times n}, B=\diag\{B_j\}_{1\le j\le N}\in\R^{n\times m}, K=[K_{ij}]_{1\le i,j\le N}\in\R^{m\times n}$ are in block form.


Wide-area control is prone to cyber attacks due to its dependence on remote measurement signals that can be compromised by a malicious adversary. In this paper, we consider DoS attacks~\cite{amin2009safe} that can result in unavailability of the signals transmitted over the attacked channels. We describe an {\em attack strategy} by $\alpha\in\{0,1\}^{N\times N}$ where entry $\alpha_{ij}=1$ means the channel from subsystem $j$ to $i$ is intact whereas $\alpha_{ij}=0$ means it is under attack. By assumption, $\alpha_{ii}=1,\forall i\in\node$. The set of all possible attack strategies is called {\em (pure) attack space} and is defined as $\A_0\triangleq\{\alpha\in\{0,1\}^{N\times N}:\alpha_{ii}=1,i\in\node\}$. The consequence of DoS attack is modeled by infinite delay of feedback signals. 

We assume that an attack strategy $\alpha$ impacts the wide-area control in the following way:
\begin{align*}
u_i=K_{ii}x_i+\sum_{j\in\node,j\neq i} \alpha_{ij}K_{ij}x_j.
\end{align*}
This corresponds to the case where the controller will ignore the component $K_{ij}x_j$ if the measurement signal of $x_j$ does not arrive within a certain time period. Such a reaction scheme is natural and commonly used in the literature~\cite{soudbakhsh2015delay}.
Now we can write the post-attack closed-loop system under attack strategy $\alpha\in\A_0$ as
\begin{align}
	\dot{x}=(A+BK\circ \alpha)x, \label{eq:clatkalpha} 
\end{align}
where $K\circ \alpha\triangleq [K_{ij}\alpha_{ij}]_{1\le i,j\le N}$ denotes the elementwise multiplication between entries of $\alpha$ (scalar) and subblocks of $K$ (matrix). Define $A(\alpha)\triangleq A+B K\circ\alpha$. To write the elementwise multiplication $\circ$ as a matrix multiplication, we consider the following transformation:
\begin{align*}
&\tilde{K}=\diag\{\tilde{K}_{[j]}\}_{1\le j\le N}\in\R^{n\times nN},\text{ where}\\
&\tilde{K}_{[j]}=\begin{bmatrix}[c|c|c|c]K_{j1} & K_{j2} & \cdots & K_{jN}\end{bmatrix}_{n_j\times n},\\
&\tilde{\alpha}=\begin{bmatrix}[c|c|c|c]\tilde{\alpha}_{[1]} & \tilde{\alpha}_{[2]} & \cdots & \tilde{\alpha}_{[N]}\end{bmatrix}^T\in\R^{nN\times n}, \text{ where} \\
&\tilde{\alpha}_{[k]}=\diag\{\alpha_{kj}\iden_{n_j}\}_{1\le j\le N}\in\R^{n\times n}.
\end{align*}
Then, $K\circ\alpha=\tilde{K}\tilde{\alpha}$. Furthermore, $\tilde{\alpha}$ can be written as the linear combination of a collection of constant matrices $\{M_{ij}\in\R^{Nn\times n}:1\le i,j\le N\}$ with entries of $\alpha$ as linear coefficients, i.e.,
\begin{align*}
&\tilde{\alpha}=\sum_{1\le i,j\le N} \alpha_{ij}M_{ij}, \text{ where} \\
&M_{ij}(p,q)=\left\{ \begin{array}{rr}
1, & \text{if } p-q=(i-1)n+\sum_{k=1}^{j-1}n_k,\\ & \text{and } q\in\{1,2,\cdots,n_j\} \\ 0, & \text{otherwise}
\end{array} \right..
\end{align*}
Now, the closed-loop system matrix $A(\alpha)$ can be written in the following form that is affine in entries of $\alpha$.
\begin{align}
&A(\alpha)\triangleq A+BK\circ \alpha=A+\sum_{1\le i,j\le N} B\tilde{K}M_{ij}\alpha_{ij}. \label{eq:Aalpha}
\end{align}

We consider a wide-area controller to be resilient if it can survive all possible (pure) attack strategies on the communication channel. 
\begin{defn}
A controller $K$ is called {\em resilient} if system~(\ref{eq:clatkalpha}) is stable for all $\alpha\in\A_0$. Conversely, it is called {\em not resilient} if there exists an $\alpha\in\A_0$ under which system~(\ref{eq:clatkalpha}) is unstable.
\label{def:resilient}
\end{defn}
In what follows, we will analyze and quantify the resilience notion given in Definition~\ref{def:resilient}. The first problem to address is under what condition the resilience of a given controller is guaranteed. We aim to derive conditions in terms of optimization problems whose structure can facilitate the analysis. A further problem is concerned with the degree of resilience. We want to define a resilience index as a normalized factor to quantify how resilient a given controller is to certain attack strategies. For the practical aspect, the goal is to develop an efficient algorithm to check the proposed resilience conditions as well as identify structural vulnerabilities.

\section{A Motivating Example}
\label{sec:ex}
WAC makes use of state information from remote buses to improve the closed-loop performance under local decentralized controllers. One may naturally think that a loss of part of remote measurement signals will only gracefully degrade closed-loop performance without causing instabilities. However, such an intuition is unfortunately not true in general. In fact, a wide-area controller can become destabilizing under a loss of a small subset of communication links. We now use a simple hypothetical example to illustrate this fact.


Consider a networked system in the form~(\ref{eq:cl}) with $N=3$ subsystems and each of which has two states and two control inputs. For simplicity, we assume there is no physical coupling among the three subsystems. Assume that $A_{11}=A_{22}=\frac{1}{2}E_2, A_{33}=E_1$, $B_1=B_2=B_2=\iden_2$, $2K_{11}=-K_{13}=K_{21}=-\frac{1}{2}K_{23}=-K_{31}=K_{33}=E_1$, $K_{12}=2K_{22}=K_{32}=E_2$, where
\begin{align*}
E_1=\begin{bmatrix}
	-3&-1\\
	12&2
	\end{bmatrix}\text{ and }
E_2=\begin{bmatrix}
	-3&1\\
	-12&2
	\end{bmatrix}.
\end{align*}
Let $A_c$ and $A_d$ be the closed-loop system matrices under controller $K$ and its full distributed realization, respectively.
\begin{align*}
&A_c\triangleq (A+BK)=\begin{bmatrix} E_1 & E_2 & -E_1\\E_1 & E_2 & -2E_1 \\ -E_1 & E_2 & 2E_1 \end{bmatrix},\\
&A_d\triangleq A+BK\circ \iden_6=\begin{bmatrix} E_1 & 0 & 0\\0 & E_2 & 0 \\ 0 & 0 & 2E_1 \end{bmatrix}.
\end{align*}
It is easy to check that both $A_c$ and $A_d$ are stable. Now consider the attack strategy $\alpha$ that targets at the communication channel from subsystem 3 to 2, i.e. $\alpha_{23}=0$. The post-attack closed-loop system matrix is 
\begin{align*}
A_a\triangleq A+BK\circ\alpha=\begin{bmatrix} E_1 & E_2 & -E_1\\E_1 & E_2 & 0 \\ -E_1 & E_2 & 2E_1 \end{bmatrix}.
\end{align*}
As $A_a$ has eigenvalues $5.1596,0.6968,-0.8631,-1.3561 \pm 6.5185i,-6.2811$, two of which are on the right half of the plane, the system is no longer stable. We can see that controller $K$ is vulnerable under the attack on the communication channel $3\to 2$.

\section{Resilience Analysis and Quantification}
In this section, we develop a Lyapunov-based framework to analyze and quantify resilience under DoS communication attacks as formulated in Section~\ref{sec:prob}.
\subsection{Resilience Conditions}
\label{sec:cond}
A system is stable if and only if all its eigenvalues have negative real part, and conversely it is unstable if and only if at least one of its eigenvalues has positive real part. Given a square matrix, we call the maximum among the real part of its eigenvalues the {\em spectral abscissa}. One direct approach for resilience condition is to first seek for the attack strategy that results in the largest spectral abscissa of closed-loop system matrix and then determine the sign of the largest spectral abscissa. For the case where it is negative, the system remains stable under all attack strategies; while for the case where it is positive, there exists at least one attack strategy that drives the system unstable. The direct formulation of resilience condition takes the following form:
\begin{align*}
    \mathbf{P0}\quad \gamma^*_0\triangleq\max_{\alpha\in\A_0}\quad & \Real(\lambda_{\max} (A(\alpha))) 
\end{align*}

If $\gamma^*_0<0$, then wide-area controller $K$ can survive all possible attacks on the communication channels, otherwise it inherits structural vulnerabilities. The optimization problem $\mathbf{P0}$ exhibits several main challenges: i) It is an unsymmetric eigenvalue problem for which the spectral theorem does not apply and thus $\lambda_{\max}$ does not have an explicit expression. ii) The objective is essentially nonconvex due to the maximization of the largest real part of eigenvalues. Typically, eigenvalue optimization problems are formulated as the minimization of the largest eigenvalue or the maximization of the smallest eigenvalue, both of which are convex. However, this is not the case for $\mathbf{P0}$. iii) The decision variable is binary and not continuous, making the problem combinatorial in nature. To address the above challenges, we next reformulate the problem via Lyapunov stability theory.

\subsubsection{A Lyapunov Formulation}
Recall that the post-attack system~(\ref{eq:clatkalpha}) is stable if and only if it admits a quadratic Lyapunov function $V(x)=x^TPx$ for some $P\succeq 0$. The condition can be written in the form of SDP: There exists a $P_0\succeq 0$ such that
\begin{align}
	A(\alpha)^TP_0+P_0A(\alpha)\prec 0. \label{eq:lyastab}
\end{align}
Conversely, the post-attack system~(\ref{eq:clatkalpha}) is unstable if and only if for all $P\succeq 0$, we can find a unit directional vector $x_P\in\{z:\|z\|=1\}$, where the subscript emphasizes the dependence of the vector on $P$, such that 
\begin{align}
	x_P^T(A(\alpha)^TP+PA(\alpha))x_P\ge 0. \label{eq:lyaunstab}
\end{align}
Inspired by the above Lyapunov characterization, we consider the following formulation:
\begin{align*}
	\mathbf{Lya0}\quad \gamma^*_{L0}\triangleq \max_{\alpha\in\A_0}\min_{P\succeq 0}\quad\lambda_{\max}(A(\alpha)^TP+PA(\alpha))
\end{align*}


\begin{thm}[Sufficient and Necessary Condition]
A controller $K$ is resilient {\em if and only if} $\gamma^*_{L0}=-\infty$, and is not resilient {\em if and only if} $\gamma^*_{L0}\ge 0$.
\label{thm:Lya0}
\end{thm}
\begin{proof}
We partition the pure attack space into two disjoint sets, i.e. $\A_0=\A_0^s\sqcup\A_0^u$, where $\A_0^s$ is the set of stabilizing attack strategies and $\A_0^u$ is the set of destabilizing attack strategies. Let $\alpha^s\in\A_0^s$. Then, system~(\ref{eq:clatkalpha}) under $\alpha^s$ is stable, that is to say there exists $P(\alpha^s)\succeq 0$ dependent on $\alpha^s$ such that $A(\alpha^s)^TP(\alpha^s)+P(\alpha^s)A(\alpha^s)\prec 0$. Then, 
\begin{align*}
&\min_{P\succeq 0} \lambda_{\max}(A(\alpha^s)^TP+PA(\alpha^s))\le \\
&\lambda_{\max}(A(\alpha^s)^TcP(\alpha^s)+cP(\alpha^s)A(\alpha^s))\to-\infty \text{ as } c\to\infty.
\end{align*}
Let $\alpha^u\in\A_0^u$. Then, system~(\ref{eq:clatkalpha}) under $\alpha^u$ is not asymptotically stable, which implies that for all $P\succeq 0$, there exists a unit directional vector $x_P\in\{z:\|z\|=1\}$ dependent on $P$ such that $x_P^T (A(\alpha^u)^TP+PA(\alpha^u)) x_P\ge 0$. Then,
\begin{align*}
&\lambda_{\max}(A(\alpha^u)^TP+PA(\alpha^u))\\
=&\max_{\|x\|=1}x^T(A(\alpha^u)^TP+PA(\alpha^u))x\\
\ge& x_P^T (A(\alpha^u)^TP+PA(\alpha^u)) x_P\ge 0,\quad\forall P\succeq 0.
\end{align*}
Thus, $\min_{P\succeq 0}\lambda_{\max}(A(\alpha^u)^TP+PA(\alpha^u))\ge 0$. 

Now, we want to show the statement for the ``resilient'' part. $(\Rightarrow)$: Assume $K$ is resilient. By Definition~\ref{def:resilient}, all the attack strategies are stabilizing, i.e. $\A_0=\A_0^s$. Thus,
\begin{align*}
\gamma_{L0}^*&=\max_{\alpha\in\A_0^s}\min_{P\succeq 0} \lambda_{\max}(A(\alpha)^TP+PA(\alpha))\\&=\max_{\alpha\in\A_0^s}-\infty=-\infty.
\end{align*}
$(\Leftarrow)$: On the other hand, if $\gamma_{L0}^*=-\infty$, then for all $\alpha\in\A_0$, $\min_{P\succeq 0} \lambda_{\max}(A(\alpha)^TP+PA(\alpha))=-\infty$, i.e. $\alpha\in\A_0^s$. Now $\A_0=\A_0^s$ and thus $K$ is resilient.

Next, we want to show the statement for the ``not resilient'' part. $(\Rightarrow)$: Assume $K$ is not resilient. By Definition~\ref{def:resilient}, $\A_0^u\neq\emptyset$. Let $\alpha^u\in\A_0^u$ be a destabilizing attack strategy. Then, 
\begin{align*}
\gamma_{L0}^*=&\max_{\alpha\in\A_0}\min_{P\succeq 0} \lambda_{\max}(A(\alpha)^TP+PA(\alpha))\\
\ge&\min_{P\succeq 0} \lambda_{\max}(A(\alpha^u)^TP+PA(\alpha^u))\ge 0.
\end{align*}
$(\Leftarrow)$: On the other hand, if $\gamma_{L0}^*\ge 0$, then there exists an $\alpha^u\in\A_0$ such that $\min_{P\succeq 0} \lambda_{\max}(A(\alpha^u)^TP+PA(\alpha^u))\ge 0$. In other words, there exists a destabilizing attack strategy and thus $K$ is not resilient.
\end{proof}

\subsubsection{A Lyapunov Relaxation}
The optimal value of $\mathbf{Lya0}$ provides an equivalent characterization of resilience as proved in Theorem~\ref{thm:Lya0}. However, the development of efficient algorithm for $\mathbf{Lya0}$ is highly nontrivial due to its binary decision variables and unbounded optimal value. For the practical use, we now consider a relaxation of $\mathbf{Lya0}$ by embedding the binary variables into closed interval $[0,1]$ and upper bounding the largest eigenvalue of positive semidefinite (P.S.D.) variable. Let $\A\triangleq\{\alpha\in[0,1]^{N\times N}:\alpha_{ii}=1,i=1,\cdots,N\}$ and $\P\triangleq\{P\in\S^n:0\preceq P\preceq \lambda_P I\}$ for some fixed $\lambda_P>0$.
\begin{align*}
	\mathbf{LyaP}\quad \gamma^*_{LP}\triangleq \max_{\alpha\in\A}\min_{P\in\P}\quad \lambda_{\max}(A(\alpha)^TP+PA(\alpha))
\end{align*}
By relaxing the feasible set for the min and constraining the one for the max, $\mathbf{LyaP}$ provides a surrogate certificate to $\mathbf{Lya0}$, which leads to a sufficient condition for resilience.

\begin{thm} A controller $K$ is resilient if $\gamma^*_{LP}<0$. Conversely, it is not resilient only if $\gamma^*_{LP}\ge 0$.
\label{thm:LyaPsuff}
\end{thm}
\begin{proof}
Since $\P\subset \{P\succeq 0\}$ and minimization over smaller set gives larger optimal value, 
\begin{align*}
g(\alpha)\triangleq&\min_{P\in\P}\lambda_{\max}(A(\alpha)^TP+PA(\alpha))\\
\ge& \min_{P\succeq 0}\lambda_{\max}(A(\alpha)^TP+PA(\alpha))\triangleq g_0(\alpha).
\end{align*}
Furthermore, $\A\supset\A_0$ and maximization over larger set gives larger optimal value,
\begin{align}
\gamma^*_{LP}=\max_{\alpha\in\A}g(\alpha)\ge \max_{\alpha\in\A_0}g(\alpha)\ge \max_{\alpha\in\A_0}g_0(\alpha)=\gamma^*_{L0}. \label{eq:lpl0}
\end{align}
For the ``if'' part, assume $\gamma^*_{LP}<0$. By relation~(\ref{eq:lpl0}), $\gamma^*_{L0}<0$. It then follows from Theorem~\ref{thm:Lya0} that $K$ is resilient. For the ``only if'' part, assume $K$ is not resilient. By Theorem~\ref{thm:Lya0}, $\gamma^*_{L0}\ge 0$. Then, $\gamma^*_{LP}\ge 0$ by relation~(\ref{eq:lpl0}). 
\end{proof}

Recall that for a symmetric matrix $M\in\S$, the largest eigenvalue of $M$ can be written as $\lambda_{\max}(M)=\min\{t: M\preceq tI\}$. Since the inner problem of $\mathbf{LyaP}$ is the minimization of the largest eigenvalue, it can be equivalently formulated in the form of SDP program. Let $g:\A\to\R$ be the optimal value of the inner minimization (over $P$) of $\mathbf{LyaP}$ defined as
\begin{align}
g(\alpha)\triangleq \min_{P\in \P}\lambda_{\max}(A(\alpha)^TP+PA(\alpha)). \label{eq:g}
\end{align}
Then for any fixed $\alpha\in\A$, $g(\alpha)$ is the optimal value of the following SDP:
\begin{align}
\begin{aligned}
g(\alpha)=\min\quad &t \\
\st\quad &A(\alpha)^TP+PA(\alpha) \preceq tI \\
&P\in\P
\end{aligned}
\label{eq:sdp}
\end{align}
Consider the following optimization problem.
\begin{align*}
	\mathbf{LyaD}\quad \gamma^*_{LD}\triangleq \min_{\alpha\in\A} \quad &t\\
	\st\quad &A(\alpha)^TP+PA(\alpha) \preceq tI \\
	& P\in\P 
\end{align*}
Note that the first constraint in $\mathbf{LyaD}$ is a Bilinear Matrix Inequality (BMI) in decision variables $P,\alpha$ and $t$. Next, we will show that the dual problem $\mathbf{LyaD}$ is equivalent to the primal problem $\mathbf{LyaP}$.
\begin{thm}
$\gamma^*_{LD}=\gamma^*_{LP}$.
\end{thm}
\begin{proof}
Let $\alpha_P^*$ be the optima of $\mathbf{LyaP}$. Then, $\gamma^*_{LP}=g(\alpha_P^*)$, for which there exists $P_P^*\in\P$ such that $A(\alpha_P^*)^TP_P^*+P_P^*A(\alpha_P^*)\preceq \gamma^*_{LP}I$. For the ``$\le$'' part, it follows from the triple $(\alpha_P^*,P_P^*,\gamma^*_{LP})$ being a feasible solution of $\mathbf{LyaD}$. For the ``$\ge$'' part, consider the BMI constraint of $\mathbf{LyaD}$. For $\alpha_P^*\in\A$, there exists $P\in\P$ such that $A(\alpha_P^*)^TP+PA(\alpha_P^*)\preceq\gamma^*_{LD}I$. By the equivalent characterization of $g(\alpha)$ given in SDP~(\ref{eq:sdp}), $g(\alpha^*_P)\le \gamma^*_{LD}$ and thus $\gamma^*_{LD}\ge\gamma^*_{LP}$.
\end{proof}

To take one step further, a natural question to ask is when the relaxed problem $\mathbf{LyaP}$ is ``exact'' in terms of resilience. In other words, whether there are cases for which solving $\mathbf{LyaP}$ results in {\em sufficient and necessary} condition. The answer is yes under some assumption. We first define {\em Lyapunov space} $\P_{\alpha}\subseteq\P$ for each pure attack strategy $\alpha\in\A_0$ as
\begin{align}
\P_{\alpha}\triangleq \{P\in\P: A(\alpha)^TP+PA(\alpha)\preceq 0, P\neq 0\}. \label{eq:lyaspace}
\end{align}
To ensure the exactness of the relaxed problem $\mathbf{LyaP}$, we require the intersection of Lyapunov spaces of any two pure attack strategies to be nonempty.
\begin{assump} 
For any $\alpha_1,\alpha_2\in\A_0$, $\P_{\alpha_1}\cap\P_{\alpha_2}\neq\emptyset$.
\label{assump:sign}
\end{assump}
The above assumption ensures the sign preserving property of the function $g$ defined in~(\ref{eq:g}) in the sense that if $g$ is strictly negative on the vertex set $\A_0$, it is strictly negative on the convex hull of $\A_0$, i.e. the relaxed attack space $\A$. On the other hand, if $g$ fails to be strictly negative on $\A$, it fails to be strictly negative on $\A_0$.
\begin{lem}
Under Assumption~\ref{assump:sign}, if $g(\alpha)<0,\forall \alpha\in\A_0$, then $g(\alpha)<0,\forall \alpha\in\A$; and conversely, if $\exists\alpha\in\A$ s.t. $g(\alpha)\ge 0$, then $\exists\alpha_0\in\A_0$ s.t. $g(\alpha_0)\ge 0$.
\label{lem:sign}
\end{lem}
\begin{proof}
Since $\A$ is a polytope with vertex set $\A_0$, it is enough to show the claim that for any $\alpha_1,\alpha_2\in\A_0,\theta\in[0,1]$, there exists $\kappa_1,\kappa_2>0$ such that 
\begin{align*}
g(\theta\alpha_1+(1-\theta)\alpha_2)\le\kappa_1g(\alpha_1)+\kappa_2g(\alpha_2).
\end{align*}
Assume that the claim holds. Consider $\alpha_{\theta}\in\A$ where $\alpha_{\theta}=\sum_{\alpha_k\in\A_0}\theta_k\alpha_k$ for some $\theta_k\in[0,1], \sum_k\theta_k=1$. If $g(\alpha_k)<0,\forall\alpha_k\in\A_0$, then $g(\alpha_{\theta})<0$. On the other hand, if $g(\alpha_{\theta})\ge 0$, then $g(\alpha_k)\ge 0$ for some $\alpha_k\in\A_0$. Now we are left to show the claim.

For the ease of notation, let $f(\alpha,P)\triangleq \lambda_{\max}(A(\alpha)^TP+PA(\alpha))$ in the rest of the proof. Let $\alpha_1,\alpha_2\in\A_0,\theta\in[0,1],P_k=\argmin_{P\in\P} f(\alpha_k,P),k=1,2$. Consider $\alpha_{\theta}=\theta\alpha_1+(1-\theta)\alpha_2$. Recall that $A(\alpha)$ defined in~(\ref{eq:Aalpha}) is affine in $\alpha$. Then, $A(\alpha_{\theta})=\theta A(\alpha_1)+(1-\theta)A(\alpha_2)$. By the convexity of $\lambda_{\max}(\cdot):\S^n\to\R$,
\begin{align*}
f(\alpha_{\theta},P)\le \theta f(\alpha_1,P)+(1-\theta)f(\alpha_2,P)\triangleq h_{\theta}(P).
\end{align*}
By assumption, $\P_{\alpha_1}\cap\P_{\alpha_2}\neq\emptyset$. Let $P_0\in\P_{\alpha_1}\cap\P_{\alpha_2}$. Since the Lyapunov space~(\ref{eq:lyaspace}) is defined by Linear Matrix Inequality (LMI), the sets $\P_{\alpha_k},k=1,2$ are convex and so is their intersection $\P_{\alpha_1}\cap\P_{\alpha_2}$. Then, $\exists t_1\in(0,1)$ s.t. $P_1'=t_1P_1+(1-t_1)P_0\in\P_{\alpha_2}$. Similarly, $\exists t_2\in(0,1)$ s.t. $P_2'=t_2P_2+(1-t_2)P_0\in\P_{\alpha_1}$. As $P_0\in\P_{\alpha_1}$, we have $f(\alpha_1,P_0)\le 0$. By the convexity of $f(\alpha,P)$ in $P$ for any fixed $\alpha$, $f(\alpha_1,P_1')\le t_1f(\alpha_1,P_1)+(1-t_1)f(\alpha_1,P_0)\le t_1f(\alpha_1,P_1)$. Similarly, $f(\alpha_2,P_2')\le t_2f(\alpha_2,P_2)$. Notice that the function $h_{\theta}:\P\to\R$ parameterized by $\theta\in[0,1]$ is the sum of two convex functions and thus is also convex. Consider $P=\beta P_1'+(1-\beta)P_2'$ for some $\beta\in[0,1]$. Then,
\begin{multline*}
h_{\theta}(P)\le\theta\beta f(\alpha_1,P_1')+\theta(1-\beta)f(\alpha_1,P_2')+\\(1-\theta)\beta f(\alpha_2,P_1')+(1-\theta)(1-\beta)f(\alpha_2,P_2').
\end{multline*}
Since $P_1'\in\P_{\alpha_2}, P_2'\in\P_{\alpha_1}$ by construction, $f(\alpha_1,P_2')\le 0$ and $f(\alpha_2,P_1')\le 0$. We prove the claim that $g(\alpha_{\theta})\le \kappa_1g(\alpha_1)+\kappa_2g(\alpha_2)$ where $\kappa_1=\theta\beta t_1$ and $\kappa_2=(1-\theta)(1-\beta)t_2$.
\end{proof}

With Lemma~\ref{lem:sign}, it is easy to obtain the following sufficient and necessary condition.
\begin{thm}[Sufficient and Necessary Condition II]
Under Assumption~\ref{assump:sign}, a controller $K$ is resilient {\em if and only if} $\gamma^*_{LP}<0$, and it is not resilient {\em if and only if} $\gamma^*_{LP}\ge 0$.
\label{thm:sign}
\end{thm}
%

\subsection{Resilience Index}
\label{sec:idx}
The conditions derived in Section~\ref{sec:cond} allow us to determine whether a given wide-area controller is resilient to all possible attack strategies. A natural additional question is how resilient the controller is to certain attack strategies. This calls for a proper definition of a normalized index to quantify the degree of resilience. Denoted by $r_K:\A_0\to[0,1]$ the resilience index of controller $K$ on the pure attack space. We consider $r_K$ to be normalized with respect to the nominal condition. In particular, $r_K$ needs to satisfy the following two conditions: i) It takes value 1 under the nominal condition when $K$ is intact, i.e. $r_K(\mathbf{1}_{N\times N})=1$; ii) It takes value 0 under destabilizing attack strategies, i.e. $r_k(\alpha)=0$ for all $\alpha\in\A_0$ under which system~(\ref{eq:clatkalpha}) is unstable.


Recall that $g:\A\to\R$ defined in~(\ref{eq:g}) is the optimal value of the inner minimization (over $P$) of the relaxed problem $\mathbf{LyaP}$. In fact, the mapping $g$ defines a performance metric for stability in the sense that for any $\alpha\in\A$, $g(\alpha)$ is the fastest decreasing rate a Lyapunov function candidate could achieve along the trajectory of $A(\alpha)$. This naturally leads to a definition of resilience index satisfying the above two conditions. Guaranteed by the design objective, the system under the nominal condition has better stability performance than the one under attack. Since the nominal condition corresponds to $\alpha=\mathbf{1}_{N\times N}$, we have i) $g(\mathbf{1}_{N\times N})\le g(\alpha),\forall\alpha\in\A_0$. On the other hand, we know from the proof of Theorem~\ref{thm:Lya0} that ii) $g(\alpha)\ge 0$ for any destabilizing $\alpha\in\A_0$. Based on i) and ii), we define resilience index $r_K:\A_0\to[0,1]$ of controller $K$ on the pure attack space $\A_0$ as follows.
\begin{align}
r_K(\alpha)=\left\{\begin{array}{ll}
0 & \text{if } g(\alpha)\ge 0\\
g(\alpha)/g(\mathbf{1}_{N\times N}) & \text{if } g(\alpha)<0
\end{array}
\right. \label{eq:index}
\end{align}
The definition in~(\ref{eq:index}) captures stability degradation of controller $K$ under different attack strategies. It is easy to see that the smaller the index $r_K(\alpha)$ is, the less resilient controller $K$ is to attack strategy $\alpha$, or in other words, the more disruption $\alpha$ will incur on $K$. For the two boundary cases, if $r_K(\alpha)=0$, controller $K$ can be destabilized by $\alpha$, while if $r_K(\alpha)=1$, $\alpha$ has no effect on controller $K$.


\section{A Path-Following Primal-Dual Algorithm}
\label{sec:algorithm}

The goal of this section is to solve the relaxed problem $\mathbf{LyaP}$. Notice that $\mathbf{LyaP}$ takes scalar continuous decision variables $\alpha_{ij},i\neq j$ and P.S.D. matrix variable $P$. By the definition of $g$ given in~(\ref{eq:g}), $\mathbf{LyaP}$ is actually the maximization of $g$ on the polytope $\A$. A natural attempt is to apply gradient ascent algorithm. The key step of gradient-based algorithm is to compute the subgradient of the objective, that is $\partial g$ for the case here. Let $f_{\alpha}(x,P)\triangleq x^T(A(\alpha)^TP+PA(\alpha))x$. 
\begin{align}
g(\alpha)=\min_{P\in \P} \max_{\|x\|=1} f_{\alpha}(x,P). \label{eq:minmax}
\end{align}
Notice that i) $x\mapsto f_{\alpha}(x,P)$ is concave and continuous for each $P$ and ii) $P\mapsto f_{\alpha}(x,P)$ is convex (actually affine) for each $x$. By the general minimax theorem, the min and the max in~(\ref{eq:minmax}) can be swapped, i.e.,
\begin{align*}
g(\alpha)&=\max_{\|x\|=1} \min_{P\in \P} f_{\alpha}(x,P)=\max_{\|x\|=1} g_x(\alpha),\text{ where} \\
g_x(\alpha)&\triangleq \min_{P\in \P} x^T(A(\alpha)^TP+PA(\alpha))x
\end{align*}
Observe that $g$ is the pointwise supremum of $g_x$ and $g_x(\alpha)$ is convex in $\alpha$ (actually affine) for each $x$. By the weak rule for pointwise supremum, a subgradient of $g$ at $\alpha$ is any element in $\partial g_{x^*(\alpha)}(\alpha)$ where $x^*(\alpha)=\argmax_{\|x\|=1} g_x(\alpha)$. Now, let's focus on computing the subgradient of $g_{x^*}$. Let $P^*(\alpha)=\argmin_{P\in \P}\lambda_{\max}(A(\alpha)^TP+PA(\alpha))$, which depends only on $\alpha$, not on $x$. Let $X^*=x^*x^{*T}$. Then,
\begin{align*}
g_{x^*}(\alpha)&=2\tr(P^*A(\alpha)X^*)\\&=2\tr(X^*P^*(A+\sum_{1\le i,j\le N} B\tilde{K}M_{ij}\alpha_{ij})).
\end{align*}
Since $g_{x^*}$ is affine in $\alpha$, the subgradient of $g_{x^*}$ coincides with the gradient taking the following form:
\begin{align*}
\partial_{ij} g_{x^*}(\alpha)=\nabla_{ij} g_{x^*}(\alpha)=2\tr(X^*P^*B\tilde{K}M_{ij}).
\end{align*}
We are now ready to introduce the primal-dual gradient ascent algorithm. 
\begin{algorithm}
\caption{Primal-dual gradient ascent algorithm}
\begin{algorithmic}[1]
\Inputs {System matrices: $A,B,K$}
\Initialize {$\alpha_{k-1}\leftarrow 1^{N\times N}$, step size $s$, tolerance $\epsilon$, $\gamma_k=-\infty,\gamma_{k-1}=0$}
\While {$\gamma_k<0$ or $\gamma_k-\gamma_{k-1}>\epsilon$}
	\State $P_k\leftarrow$ optimality of $\mathbf{LyaD}$ with $\alpha=\alpha_{k-1}$ \Comment{Update dual variable $P$: SDP with LMI constraints}
	\State $x_k\leftarrow$ eigenvector associated with the largest eigenvalue of $A(\alpha_{k-1})^TP_k+P_kA(\alpha_{k-1})$, $X_k\leftarrow x_kx_k^T$
	\State $\eta_{ij}\leftarrow \tr(X_kP_kB\tilde{K}M_{ij}),\eta\leftarrow \eta/\|\eta\|_F$ \Comment{Compute gradient $\nabla g(\alpha_{k-1})$}
	\State $\alpha_k\leftarrow\alpha_{k-1}+s\eta$ \Comment{Update primal variable $\alpha$: gradient ascent}
	\State $\alpha_k\leftarrow\Pi_{\A}(\alpha_k)$ \Comment{Project $\alpha_k$ onto relaxed attack set}
	\State $\gamma_{k-1}\leftarrow\gamma_k,\gamma_k\leftarrow x_k^T(A(\alpha_k)^TP_k+P_kA(\alpha_k))x_k$ \Comment{Compute objective}
	\State $\alpha_{k-1}\leftarrow\alpha_k$ 
\EndWhile
\Outputs {optimality $\gamma_k,\alpha_k$}
\end{algorithmic}
\label{algo:pd}
\end{algorithm}

Let $\{\gamma_k\}_{k\in\N}$ be the sequence of optimal value and $\{\alpha_k\}_{k\in\N}$ be the sequence of optima returned by Algorithm~\ref{algo:pd}.
\begin{thm}
A controller $K$ is resilient if $\gamma_k\uparrow\gamma^*<0$. Conversely, it is not resilient only if $\gamma_k\uparrow 0$. 
\label{thm:cvgce}
\end{thm}
\begin{proof}
Given $\alpha_{k-1}$, $P_k$ is the optima of $\mathbf{LyaD}$ for $\alpha=\alpha_{k-1}$ s.t. $P_k=P^*(\alpha_{k-1})$, where
\begin{align*}
P^*(\alpha)=\argmin_{P\in\P}\lambda_{\max}(A(\alpha)^TP+PA(\alpha)).
\end{align*}
Now given $\alpha_{k-1}$ and $P_k$, $x_k$ is the eigenvector associated with the largest eigenvalue of $A(\alpha_{k-1})^TP_k+P_kA(\alpha_{k-1})$.
\begin{align*}
x_k=\argmax_{\|x\|=1}x^T(A(\alpha_{k-1})^TP_k+P_kA(\alpha_{k-1}))x.
\end{align*}
To evaluate the subgradient of $g$, we define a collection of functions $g_x:\A\to\R$ parameterized by $x\in\{z:\|z\|=1\}$.
\begin{align*}
g_x(\alpha;P^*(\alpha))\triangleq x^T(A(\alpha)^TP^*(\alpha)+P^*(\alpha)A(\alpha))x.
\end{align*}
Observe that $g(\cdot)$ is the pointwise maximum of $g_x(\cdot;\cdot)$ where the second variable is determined by the first variable and is uniform in $x$. By the weak rule for pointwise supremum, a subgradient of $g$ at $\alpha$ is any element in $\partial g_{x^*}(\alpha)$ where $x^*$ is such that $g(\alpha)=g_{x^*}(\alpha)$. For $\alpha=\alpha_{k-1}$, we have $g(\alpha_{k-1})=g_{x_k}(\alpha_{k-1};P_k)$ and thus
\begin{align*}
\partial g(\alpha_{k-1}) \ni \partial g_{x_k}(\alpha_{k-1};P_k).
\end{align*}
Due to $g_x(\cdot;\cdot)$ is affine in the first variable, $\partial g_x=\nabla g_x$. Let $\eta=\nabla g_{x_k}(\alpha_{k-1};P_k)\in\R^{N\times N}$. Then, $\eta\in\partial g(\alpha_{k-1})$. By the property of subgradient, for $s>0$ small enough,
\begin{align*}
g(\alpha_k)=g(\alpha_{k-1}+s\eta)\ge g(\alpha_{k-1})+s\langle\eta, \Pi_{\mathcal{T}_{\A}(\alpha_{k-1})}(\eta)\rangle,
\end{align*}
where $\mathcal{T}_{\A}(\alpha)$ denotes the tangent cone of $\A$ at $\alpha$ and $\Pi_{\mathcal{M}}(\cdot)$ denotes the projection operator onto $\mathcal{M}$. For $\alpha\in\interior(\A)$, $\Pi_{\mathcal{T}_{\A}(\alpha)}(\eta)=\eta,\forall \eta\in\R^n$. For $\alpha\in\partial(\A)$, $0\le \langle\eta, \Pi_{\mathcal{T}_{\A}(\alpha_{k-1})}(\eta)\rangle<\|\eta\|^2$. Thus,
\begin{align*}
\gamma_k=g(\alpha_k)\ge g(\alpha_{k-1})=\gamma_{k-1},\forall k\in\N.
\end{align*}
Now that the sequence $\{\gamma_k\}_{k\in\N}$ is increasing and upper bounded by 0, the rest of the proof follows from Theorem~\ref{thm:LyaPsuff}.
\end{proof}

\section{Simulation Results}
\label{sec:sim}
In this section, we illustrate the proposed resilience framework on the IEEE 39-bus system~\cite{rogers2012power}. To obtain the linearized model of the form~(\ref{eq:cl}), an object-oriented version of PST has been used~\cite{chow1992toolbox}. There are $N=10$ buses in the network-reduced model where bus 1 represents subtransient salient pole with $n_1=7$ states, bus 2-9 represent subtransient round rotor with $n_i=8$ states for $i=2,\cdots,9$ and bus 10 represents subtransient round rotor with $n_{10}=4$ states. Each bus from 1 to 9 has a scalar wide-area control input, i.e. $m_i=1, i=1,\cdots,9$ and bus 10 has no control, i.e. $m_{10}=0$. The overall system has $n=75$ states and $m=9$ control inputs. The dimension of system matrices are summarized are follows: $A\in\R^{75\times 75},B\in\R^{75\times 9},K\in\R^{9\times 75}$.

We consider two wide-area controllers $K_1,K_2\in\R^{9\times 75}$ that are relatively centralized as compared with the sparse-promoting controller $K_{sp}$ given in~\cite{dorfler2014sparsity}. The spectral abscissas (maximal real part of eigenvalues) of closed-loop system under the three controllers are summarized in Table~\ref{tbl:eig}. We can see that $K_1,K_2$ have better closed loop performance than $K_{sp}$ since the former two leverage more remote state information than the latter. However, the better closed-loop performance comes at the price of exposing vulnerabilities to cyber attacks. Next, we will analyze the resilience of $K_1,K_2$ under attacks on the communication channels using the proposed framework. 

We first give an overview on the resilience of the two controllers. In particular, we enumerate all possible single- and double-channel attack strategies and summarize the worst attack strategy of each scenario in Table~\ref{tbl:atk}. We can see that $K_1$ is resilient to all the 81 single-channel attack strategies, among which the worst attack 10$\to$2 still results in negative spectral abscissa -0.1744. On the other hand, $K_2$ is not resilient to single-channel attack and there are 2 out of 81 single-channel attack strategies that can destabilize the system. Furthermore, neither $K_1$ nor $K_2$ is resilient to double-channel attack. But $K_1$ is relatively more resilient than $K_2$ as $K_1$ has much less destabilizing double-channel attack strategies (total of 4) than $K_2$ (total of 167). Overall, $K_1$ is more resilient than $K_2$. In what follows, we quantify and analyze the resilience under cyber attacks of the two controllers by first computing their resilience indices and then identifying critical channels based on the machinery we developed in this paper.

\begin{table}[!ht]
\caption{spectral abscissa of closed-loop system}
\centering
\begin{tabular}{@{}lllll@{}}\toprule
& w/o feedback & w/ $K_1$ & w/ $K_2$ & w/ $K_{sp}$\\ \midrule
$\max_i\Real(\lambda_i)$ & -4.9523e-06 & -0.19184 & -0.19195 & -5.8433e-02 \\ 
\bottomrule
\end{tabular}
\label{tbl:eig}
\end{table}

\begin{table}[!ht]
\caption{single- and double-channel attack}
\resizebox{1\linewidth}{!}{
\begin{tabular}{@{}ccccccccc@{}}\toprule
& \multicolumn{2}{c}{total \# of destab.} &   & \multicolumn{2}{c}{worst attack} &   & \multicolumn{2}{c}{max spec. abs.} \\
\cmidrule{2-3} \cmidrule{5-6} \cmidrule{8-9} 
		& 1-ch	& 2-ch 		& & 1-ch 		& 2-ch & & 1-ch & 2-ch \\\midrule
$K_1$ 	& 0/81	& 4/3240	& & $10\to 2$	& {\small $\begin{array}{c}5\to4\\6\to4\end{array}$} & & -0.1744 & 0.1268 \\
$K_2$ 	& 2/81	& 167/3240	& & $5\to 4$	& {\small $\begin{array}{c}4\to1\\5\to4\end{array}$} & & 0.1484 & 0.6332 \\
\bottomrule
\end{tabular}
}
\label{tbl:atk}
\end{table}


\begin{figure}[!ht]
	\centering
	\includegraphics[width=1\linewidth]{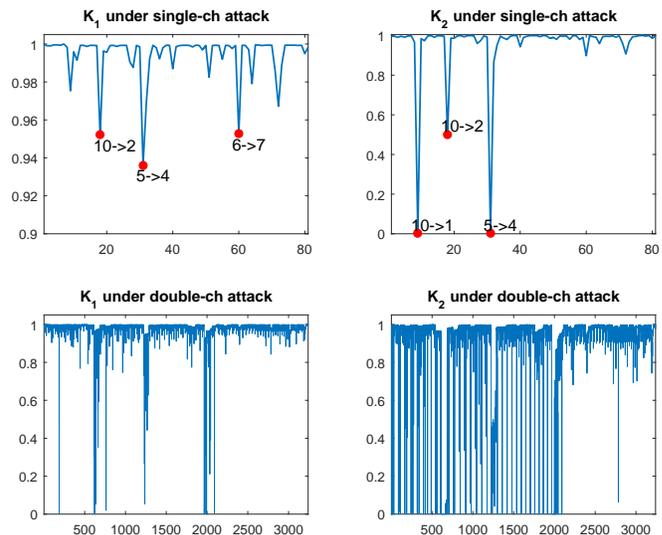}
	\caption{resilience index under single- and double-channel attack}
	\label{fig:index}
\end{figure}
\begin{figure}[!ht]
	\centering
	\includegraphics[width=1\linewidth]{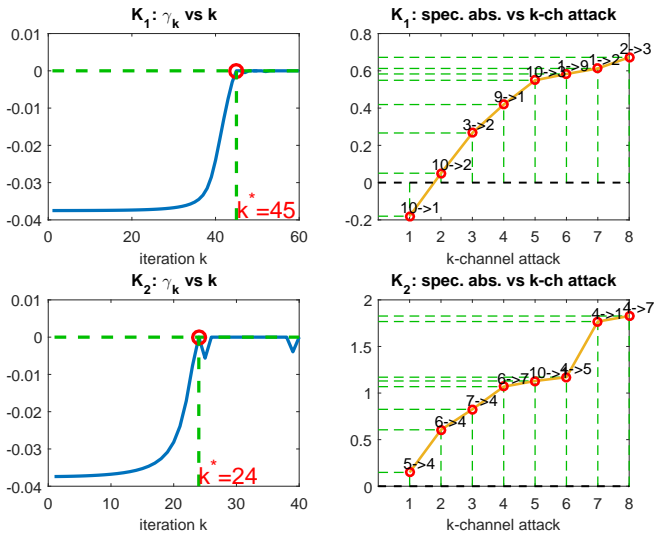}
	\caption{left: convergence of Algorithm~\ref{algo:pd}, right: spectral abscissa under $k$-channel attack}
	\label{fig:algo}
\end{figure}


\subsection{Resilience Index}
We compute resilience indices of the two controllers under single- and double-channel attack using the definition given in~(\ref{eq:index}) and present them in Fig.~\ref{fig:index}. The worst three single-channel attack strategies, corresponding to the smallest three resilience indices, are highlighted by red dots. We can see that resilience index of $K_1$ is larger than that of $K_2$, suggesting $K_1$ is more resilient than $K_2$, as what is expected. This shows that our resilience index is an effective metric to quantify resilience.

\subsection{Identification of Critical Channels}
We apply Algorithm~\ref{algo:pd} to check the resilience criterion for the two controllers. The sequences of optimal value are plotted in the left panel of Fig.~\ref{fig:algo}. We can see that $\gamma_k\uparrow 0$ in both cases. By Theorem~\ref{thm:cvgce}, we know that $K_1$ and $K_2$ both satisfy the necessary condition for non-resilience. To identify critical channels, we focus on the optimal relaxed strategy $\alpha^*$ obtained at the instance $k^*$ when the optimal value firstly reaches 0. We rank the criticality of channels by the magnitude of their corresponding entry of $\alpha^*$, that is the smaller $\alpha^*_{ij}$ is, the more critical channel $j\to i$ is. We consider $k$-channel attacks for $k=1,\cdots,8$ generated by the criticality ranking and plot the resulting spectral abscissa on the right panel of Fig.~\ref{fig:algo}. The $k$-th most critical channel is labeled on top of the red circle corresponding to $k$-channel attack, whose attack set includes the first $k$ most critical channels. We can see that the system is driven more and more unstable under the sequence of critical $k$-channel attack strategy. Therefore, we successfully identify structural vulnerabilities by the criticality ranking.

%

\section{Conclusion}
This paper proposes a novel framework for resilience analysis and quantification of wide-area control of power systems. We formally define the notion of resilience in the presence of cyber attacks. Resilience conditions are given in terms of Lyapunov-based optimization problems. A resilience index is defined to quantify the degree of resilience. We develop an efficient numerical algorithm to check the proposed resilience criterion as well as identify structural vulnerabilities.

\bibliographystyle{IEEEtran}
\bibliography{resilientControl}
\end{document}